\documentclass[12pt]{amsart}
\usepackage{amsmath}
\usepackage[scaled]{helvet} 
\it
\usepackage{amssymb}
\usepackage{graphicx}

\usepackage{amsmath, amsfonts,graphicx,fullpage,etex,tikz,hyperref}

\usepackage{enumerate}

\newtheorem{theorem}{Theorem}[section]
\newtheorem{lemma}[theorem]{Lemma}

\newtheorem{corollary}[theorem]{Corollary}
\newtheorem{conjecture}[theorem]{Conjecture}

\theoremstyle{definition}
\newtheorem{definition}[theorem]{Definition}
\newtheorem{question}{Question}

\theoremstyle{remark}
\newtheorem{remark}[theorem]{Remark}

\numberwithin{equation}{section}

\newcommand{\Aff}{\mathrm{Aff}}
\newcommand{\Aut}{\mathrm{Aut}}

\newcommand{\id}{\mathrm{Id}}
\newcommand{\Z}{\mathbb{Z}}

\newcommand{\R}{\mathbb{R}}

\newcommand{\mc}{\mathcal}
\newcommand{\set}[1]{\left\lbrace #1 \right\rbrace}

\newcommand{\abs}[1]{\left| #1 \right|}
\newcommand{\norm}[1]{{\left|\left| #1 \right|\right|}}
\newcommand{\ve}{\varepsilon}
\newcommand{\of}{\circ}

\definecolor{darkcyan}{rgb}{0. 0.65, 0.65}

\def\h-rk{rk ^{h}}

\newtheorem{Structural Stability Theorem}[theorem]{Structural Stability Theorem}

\def\bt{\begin{theorem}}
\def\et{\end{theorem}}
\def\bd{\begin{definition}}
\def\ed{\end{definition}}
\def\bl{\begin{lemma}}
\def\el{\end{lemma}}

\def\Stab{\operatorname{Stab}}

\title{Instability for rank one factors of product actions}
\author{Kurt Vinhage}

\begin{document}

\maketitle

\begin{abstract}
We provide a counterexample to a standard interpretation of the Katok-Spatzier conjecture, and pose questions which may serve as reasonable replacements.
\end{abstract}

\section{Introduction}

The Katok-Spatzier conjecture for higher-rank abelian group actions without rank one factors can be traced back to the work of Burns and Spatzier on compact higher-rank Riemannian manifolds \cite{burns-spatzier}, as well as works of Hurder \cite{hurder90}, which were extended by Katok and Lewis \cite{katok-lewis1,katok-lewis2} for actions of higher-rank lattices. The key ideas in the proofs of these rigidity results were associated higher-rank abelian group actions, and their hyperbolicity properties.

These ideas lead to a series of papers in the 90's, where Katok and Spatzier gave several striking features for irreducible actions on tori and some Weyl chamber flows without rank one factors, including measure and cocycle rigidity \cite{ks-meas,ks-coc}.

The third prototypical result in the rigidity program provided the basis for the Katok-Spatzier conjecture: smooth local rigidity of some natural higher-rank actions without rank one factors \cite{ks-local}, which showed that any $C^\infty$, sufficiently $C^1$-close perturbation was $C^\infty$ conjugate to the original action after a linear change of coordinates. For a more complete history of the rigidity program of higher rank abelian and semisimple Lie group actions, see, \cite{spatzier-survey}, which discusses rigidity phenomena very broadly, \cite{fisher-survey}, which is focused on the context of lattices in semisimple Lie groups, or the introduction to \cite{spatzier-vinhage}, which focuses on the history of rigidity program for abelian actions.

Analysis of several of the proofs reveals a similar theme: obtaining {\it isometric} behavior on certain dynamically defined foliations coming from {\it hyperbolic} behavior allows one to spread invariant structures around. The mixing of these conventional opposites, uniform hyperbolicity and isometric behavior, leads to rigidity.


 It is therefore natural to establish two critical assumptions. First, that the action is {\it Anosov}, which provides the hyperbolicity assumptions to obtain dynamically-defined foliations, which in algebraic examples have algebraic structure.

\begin{definition}
\label{def:anosov}
Let $\alpha : \R^k \curvearrowright X$ be a locally free $C^r$ group action on a $C^\infty$ manifold $X$ and $\mc O$ denote its orbit foliation. We say that $a \in \R^k$ is an {\it Anosov element} if there is an $\R^k$-invariant splitting of the tangent space $TX = E^u_a \oplus E^s_a \oplus T\mc O$ {  into nontrivial subbundles}, some $\lambda, C > 0$ such that for all $t > 0$,

\[ \norm{d\alpha(ta)|_{E^s_a}} \le Ce^{-\lambda t} \mbox{ and } \norm{d\alpha(-ta)|_{E^u_a}} \le Ce^{-\lambda t}. \]

 We say that the action is {\it Anosov} if it has at least one Anosov element. We say that an action is {\it totally Anosov} if the set of Anosov elements are dense. 
 
 An Anosov action $\alpha$ is {\it transitive} if there exists a point with a dense $\R^k$-orbit. We say that an Anosov action is {\it cone transitive} if there exists an open cone $C \subset \R^k$ and $x \in X$ such that { $\alpha(C)x$} is dense, and the only non-Anosov element of $\overline{C}$ is 0.
\end{definition}

Second, we need a way to rule out well-known perturbative families in the setting of Anosov flows and diffeomorphisms. The following definition does this by saying that no factor of the action is a flow or diffeomorphism.

\begin{definition}
If {  $r \ge 1$ and} $\alpha : \R^k \times \Z^\ell \curvearrowright X$ is a locally free action, a {\it $C^r$ rank one factor} of $\alpha$ { consists of the following data:

\begin{itemize}
\item  a $C^\infty$ manifold $Y$ with $\dim(Y) \ge 1$,
\item a $C^r$, fixed-point free flow $\psi_t : Y \to Y$ or diffeomorphism $f : Y \to Y$,
\item  a $C^r$ submersion $\pi : X \to Y$, and
\item a surjective homomorphism $\sigma : \R^k \times \Z^\ell \to \R$ (or $\sigma : \R^k \times \Z^\ell \to \Z$) such that
\[ \pi(\alpha(a)x) = \psi_{\sigma(a)}\pi(x) \qquad \mbox{   or   }\qquad \pi(\alpha(a)x) = f^{\sigma(a)}\pi(x).\]
\end{itemize}
}

We allow for passing to a finite index subgroup of $\R^k \times \Z^\ell$ or a finite cover of $X$. 
\end{definition}

When the action is homogeneous, it is more clear what is meant by a rank one factor. In contrast, the definition of a rank one factor in a more general setting has been unclear and nebuluous throughout the development of the theory. It is usually used to guarantee some transitivity or ergodicity of actions of subactions (see, \cite[Theorem 2.1]{spatzier-vinhage}, Section \ref{sec:rankone-hyperplanes} and Lemma \ref{lem:lines-ergodic}).

Finally, we need to identify the models for such actions. The following definition includes the two common ``building blocks'' for Anosov $\R^k \times \Z^\ell$ actions: Weyl chamber flows and actions by toral automorphisms. It is closed under taking products, suspsensions and skew products, so is the natural class to consider.

\begin{definition}
An {\it algebraic action} of $\R^k \times \Z^\ell$ is constructed from the following data:

\begin{itemize}
\item a Lie group $G$, 
\item a compact subgroup $M \subset G$, 
\item a (cocompact) lattice $\Gamma \subset G$, and 
\item a homomorphism $i : \R^k \times \Z^\ell \to \Aff_{M,\Gamma}(G)$.
\end{itemize}

Here $\Aff_{M,\Gamma}(G)$ is the group of affine maps $g \mapsto h\varphi(g)$, where $h \in Z_G(M)$, the centralizer of $M$ in  $G$, and $\varphi \in \Aut(G)$ is such that $\varphi$ preserves $M$, $Z_G(M)$ and $\Gamma$. We denote the image of $a$ under $i$ by $i_a$. The action is defined by $\alpha : \R^k \times \Z^\ell \curvearrowright X = M \backslash G / \Gamma$, where

\[ \alpha(a)  Mg\Gamma = Mi_a(g)\Gamma. \]
\end{definition}

One may sometimes expect a topological orbit equivalence to such models in rank one (as in the Smale conjecture for Anosov diffeomorphisms and associated Franks-Manning theorem on tori and nilmanifolds), but usually not a conjugacy. Such topological rigidity fails in the case of Anosov flows, which have several constructions which change the topological orbit structure significantly.

With these definitions in hand, and the proof of local rigidity using them as the ``essential'' tools to obtain rigidity, the following conjecture was formulated. { This conjecture was well-circulated in the 1990s but not written down precisely at the time. Formulations can be found, for instance in \cite[Conjecture 16.8]{hass07}, \cite[Section 5]{katok-collected}, and \cite[(2.4)]{burns-katok}:}

\begin{conjecture}[Katok-Spatzier]
\label{conj:kat-spa}
If $\R^k \times \Z^\ell \curvearrowright X$ is a transitive, $C^\infty$, Anosov action on a compact manifold without $C^\infty$ rank one factors, then (up to finite cover) it is $C^\infty$ conjugate to an algebraic action.
\end{conjecture}


Progress toward the conjecture in special cases has been made incrementally over the last 20 years. The optimal results for $\Z^k$-actions were obtained by Rodriguez-Hertz and Wang, who showed the conjecture for actions on nilmanifolds and tori in \cite{RH-W}, and for $\R^k$-actions, the author and Spatzier proved the conjecture for cone transitive, totally Cartan actions (see Definition \ref{def:cartan}) \cite{spatzier-vinhage}.

\begin{remark}
\label{rmk:params-matter}
The Katok-Spatzier rigidity program is meant to promise rigidity of smooth structures {\it and} parameterizations of orbits. In particular, the cocycle rigidity results for genuinely higher-rank actions means that one may not take a nontrivial time change of a homogeneous action without rank one factors, so an important feature of the rigidity program for higher-rank actions has been that the actions are considered, and not just their orbit foliations. We expand on this remark in Section \ref{sec:updates}.
\end{remark}

The main theorem of this paper provides a family of counterexamples to the conjecture. We will define {\it continuously accessible} in Definition \ref{def:accessibility}, but note here that it includes all contact Anosov flows, in particular all geodesic flows for surfaces of negative curvature.

\begin{theorem}
Let $f_s : Y_1 \to Y_1$ and $g_t : Y_2 \to Y_2$ be continuously accessible $C^\infty$ Anosov flows on 3-manifolds. Then there exists a $C^\infty$, cone transitive action of $\R^2$ on $X = Y_1 \times Y_2$ which is Anosov, has no $C^1$ rank one factors and is not homogeneous.
\end{theorem}

In Section \ref{sec:updates}, we will comment on features of this family of examples and how a revision to Conjecture \ref{conj:kat-spa} could be formulated to accomodate these new examples.

\hspace{1cm}

{\noindent \bf Acknowledgements.} The author would like to thank Ralf Spatzier for his encouragement, collaborations, and frequent discussions on abelian group actions for several years. The author would also like to thank the Simons Center for Geometry and Physics for a workshop hosted in March 2022, where this class of examples was discovered, {  as well as the referee of the paper, who pointed out several imprecise treatments and made suggestions to improve the readability of this paper}.

\section{Cocycles and time changes of abelian group actions}

The central idea of this paper is to use a reparameterization of the $\R^2$-orbits of a product action to destroy the product structure of the action. This is analogous to taking a time change of a flow, so we call such perturbations time changes of an $\R^2$-action.

\begin{definition}
If $\alpha_0 : \R^k \curvearrowright X$ is a { locally free} action of $\R^k$ on a space $X$, a {\it $C^r$-time change} of $\alpha_0$ is an action $\alpha : \R^k \curvearrowright X$ such that there exists a $C^r$ map $\varphi :  \R^k \times X \to \R^k$ satisfying { $\varphi(0,x)= 0$ for all $x \in X$ and}
\begin{equation}
\label{eq:determining-function}
 \alpha(a)x = \alpha_0(\varphi(a,x))x,
 \end{equation}
and for every $x \in X$, $\varphi(\cdot ,x)$ is a $C^r$ diffeomorphism from $\R^k \to \R^k$. We say that $\varphi : \R^k \times X \to \R^k$ {\it determines} $\alpha$.
\end{definition}

Not every function $\varphi$ will determine a time change, so we must be careful in constructing it. The main tool for doing so is the following.

\begin{definition}
If $\alpha : \R^k \curvearrowright X$ is an action of $\R^k$ on a space $X$, an {\it (abelian) cocyle} over $\alpha$ is a map $\beta : \R^k \times X \to \R^\ell$ such that

\begin{equation}
\label{eq:cocycle-eqn}
\beta(a+b,x) = \beta(a,x) + \beta(b,\alpha(a)x).
\end{equation}

A cocycle is a {\it coboundary} if there exists some $H : X \to \R^\ell$ such that $\beta(a,x) = H({ \alpha(a)}x) - H(x)$. We consider cocycles and coboundaries in the $C^\infty$, $C^r$, $C^0$ and measurable categories when appropriate. 
\end{definition}

The cocycle property will help us determine which functions $\varphi : \R^k \times X \to \R^k$ determine an $\R^k$-action via the formula \eqref{eq:determining-function}. Indeed, while such a function $\varphi$ always reparameterizes orbits, it must satisfy the cocycle property over the new candidate action $\alpha$ to determine a time change.

\begin{lemma}
\label{lem:time-change-cocycle}
If $\alpha$ is a time change of $\alpha_0$ with determining function $\varphi$, then $\varphi$ is a cocycle over $\alpha$.
\end{lemma}

\begin{proof}
We verify \eqref{eq:cocycle-eqn} directly from the condition that $\alpha$ is an action of $\R^k$:

\begin{multline*} \alpha_0(\varphi(a+b,x))x = \alpha(a+b)x = \alpha(b)\alpha(a)x = \alpha_0(\varphi(b,\alpha(a)x))\alpha(a)x =  \\ \alpha_0(\varphi(b,\alpha(a)x))\alpha_0(\varphi(a,x))x = \alpha_0(\varphi(b,\alpha(a)x) + \varphi(a,x))x.
\end{multline*}

Since the action is locally free, {  $\varphi(0,x) = 0$, and $\varphi$ is continuous}, we conclude the cocycle equation for small values of $\R^k$, and hence for large values of $\R^k$ by writing them as integer multiples of small values and applying the cocycle equation the correct number of times.
\end{proof}

If we wish to construct a time change of an action $\alpha_0$ from a cocycle $\beta$ over $\alpha_0$, Lemma \ref{lem:time-change-cocycle} suggests that we interchange the roles of which one is a time change of the other. In particular, we may think of $\alpha$ as the original action and $\alpha_0$ as the time change, so that the determining function is a cocycle over $\alpha_0$. The cost is that we must be able to invert the function $\beta$ as a map from $\R^2$ to $\R^2$ with fixed $x$. This is formalized in the following lemma.

\begin{lemma}
\label{lem:build-time-change}
There exists $\ve_0 > 0$ with the following property: Let $\alpha_0 : \R^k \curvearrowright X$ be a $C^\infty$ action of $\R^k$ and $\beta : \R^k \times X \to \R^k$ be a $C^\infty$ cocycle over $\alpha_0$ such that $\norm{d_a\beta(0,x) - \id} \le \ve_0$ for every $x \in X$, where $d_a$ represents the derivative of $\beta$ is the $\R^k$ coordinate. Then there exists a $C^\infty$ function $\varphi : \R^k \times X \to { \R^k}$ such that:

\begin{enumerate}
\item $\varphi(\beta(a,x),x) = a = \beta(\varphi(a,x),x)$ for every $a \in \R^k$, $x \in X$,
\item $\varphi(\cdot, x)$ is a $C^\infty$ diffeomorphism from $\R^k$ to $\R^k$ {  for all $x \in X$}, and
\item $\varphi$ determines a $C^\infty$ time change of $\alpha_0$.
\end{enumerate}
\end{lemma}

\begin{proof}
We first construct the function $\varphi$ at a fixed $x$ by showing the map $\beta( \cdot,x) : \R^k \to \R^k$ has a global inverse. Indeed, by picking $\ve_0$ sufficiently small, we may assume that $d_{ a}\beta(a,0)$ is invertible and hence that there is a local inverse to the function $\beta(\cdot,x)$ defined on a neighborhood $B(0,\eta_x) \subset \R^2$ for some $\eta_x > 0$, and the inverse function is $C^\infty$. 
To see that it has a global inverse, notice that by the cocycle equation gives that 

\begin{equation}
\label{eq:extending-beta}
\beta(a+b,x) = \beta(a,\alpha_0(b)x) + \beta(b,x).
\end{equation}

 By fixing $b$ and letting $a$ vary, we get that $d_a\beta(b,x) = d_a\beta(0,\alpha_0(b)x)$, so $d_a\beta(b,x)$ is close to the identity for all $b \in \R^k$ as well. In particular, $\eta_x$ can be chosen uniformly in $x$ {  as it can be estimated on the closeness of $d_a\beta(\cdot,x)$ to the identity. Therefore, function $\beta$ is surjective   since its image can always be extended by a ball of uniform size using \eqref{eq:extending-beta}.}

To see that it is globally injective, note that integrating the closeness of the derivative yields that $\norm{\beta(a,x)-a} \le \ve_0\norm{a}$ for all $ a \in \R^k$, $x \in X$. If $\beta(a,x) = \beta(b,x)$, then $\beta(a,x) - \beta(b,x) = \beta(a-b,\alpha_0(b)x) = 0$. But $\norm{\beta(a-b,\alpha_0(b)x} \ge (1-\ve_0)\norm{a-b}$, so this is not possile unless $a = b$. Therefore, for a fixed $x$, the map $\beta$ has a global $C^\infty$ inverse in the coordinate $a$, which we denote by $\varphi$.

To see that $\varphi$ determines a time change, we check that $\alpha(a)x := \alpha_0(\varphi(a,x))x$ is an abelian action:

\begin{multline*} \alpha(a)(\alpha(b)x) = \alpha(a)( \alpha_0(\varphi(b,x))x ) = \alpha_0(\varphi(a,\alpha_0(\varphi(b,x))x))\alpha_0(\varphi(b,x))x \\
 = \alpha_0( \varphi(a,\alpha_0(\varphi(b,x))x)+\varphi(b,x)) x.
 \end{multline*}

We therefore need to check that $\varphi(a+b,x) = \varphi(a,\alpha_0(\varphi(b,x))x)+\varphi(b,x)$. Since we have shown that $\beta$ is invertible in the $a$ coordinate, it suffices to check equality { after} applying $\beta(\cdot,x)$ to each side. Then the desired equality follows exactly from the cocycle equation for $\beta$ over $\alpha_0$

\begin{multline*} \beta(\varphi(b,x) +  \varphi(a,\alpha_0(\varphi(b,x))x),x) = \beta(\varphi(b,x),x) + \beta(\varphi(a,\alpha_0(\varphi(b,x))x),\alpha_0(\varphi(b,x))x) \\ = b + a = \beta(\varphi(a+b,x),x).
\end{multline*}

Therefore, $\varphi$ determines an $\R^k$ group action $\alpha$ which is a time change of $\alpha_0$. It is clear from the definition that $\alpha$ is a $C^\infty$ group action, since $\beta$ is assumed to be $C^\infty$ in all coordinates, and the derivatives of $\varphi$ can be computed explicitly from the definition.
\end{proof}

\section{Anosov actions and coarse Lyapunov foliations}
\label{sec:coarse-lyapunov}

We now summarize the theory of coarse Lyapunov foliations, for details see \cite[Section 4.1]{spatzier-vinhage}. Given an Anosov action, through standard constructions from the theory of normal hyperbolicity theory, each Anosov element has a pair of H\"older foliations $W^s_a$ and $W^u_a$ with $C^r$ leaves. $W^*_a$ are unique integral foliations of the distributions $E^*_a$, $* = s,u$. By considering the action of other elements on such foliations, one may refine them to find {\it common stable manifolds} $W^s_{a_1,\dots,a_n}$ for any collection of Anosov elements $a_1,\dots,a_n$, which are characterized as

\[ W^s_{a_1,\dots,a_n}(x) = \set{ y \in X : d(\alpha(ka_i)x,\alpha(ka_i)y) \xrightarrow{k \to \infty} 0 \mbox{ for }i=1,\dots,n}.\]

{  For a fixed colleciton $a_1,\dots,a_n$, the common stable manifolds form a H\"older foliation with $C^r$ leaves, with corresponding distriubution} $TW^s_{a_1,\dots,a_n} = \bigcap_{i=1}^n E^s_{a_i}$. {  We call the corresponding foliation the common stable foliation determined by $a_1,\dots,a_n$.}

\begin{definition}
\label{def:cartan}
A common stable { foliation $\set{W^\beta(x) = W^s_{a_1,\dots,a_n}(x) : x \in X}$} is a {\it coarse Lyapunov foliation} of an action $\alpha : \R^k \curvearrowright X$ if for any Anosov element $a \in \R^k$, $W^\beta \subset W^s_a$ or $W^\beta \subset W^u_a$. We call $E^\beta = TW^\beta$ the corresponding {\it coarse Lyapunov distribution}. Let $\Delta$ denote an indexing set for the collection of coarse Lyapunov foliations.

We say that an Anosov action is {\it Cartan} if for every $\beta \in \Delta$, $\dim(W^\beta) = 1$. We say that an action is {\it totally Cartan} if it is Cartan and totally Anosov.
\end{definition}


Call a point $x \in X$ {\it $\R^k$-periodic} for an $\R^k$-action $\alpha$ if $\alpha(\R^k)x$ is closed. Since all actions are assumed to be locally free, this implies that the $\alpha(\R^k)x$ is diffeomorhpic to $\mathbb{T}^k$, and that $\Stab(x)$ is a lattice in $\R^k$. The following can be found in \cite[Lemma 4.17]{spatzier-vinhage}. While it is stated there for totally Anosov actions, the totally modifier is required for the other list of equivalences.

\begin{lemma}
\label{lem:dense-periodic}
Let $\alpha : \R^k \curvearrowright X$ be a cone transitive, Anosov action. Then the set of $\R^k$-periodic points is dense.
\end{lemma}

\begin{lemma}
\label{lem:invariant-bundles}
If $\alpha : \R^k \curvearrowright X$ is a { cone transitive} Anosov $C^r$ group action on a $C^\infty$ manifold $X$, then $TX = T\mc O \oplus \bigoplus_{\beta \in \Delta} W^\beta$. If $\alpha$ is Cartan, and $V \subset TX$ is a {  continuous} $\R^k$-invariant distribution, then there exists a subset $\Phi \subset \Delta$ and a subbundle $V_{\mc O} \subset T \mc O$ such that $V = V_{\mc O} \bigoplus_{\beta \in \Phi} E^\beta$.
\end{lemma}

\begin{proof}
The part of the lemma for Anosov actions follows from \cite[Corollary 4.6]{spatzier-vinhage}.

{ Now assume the Cartan condition,} and let $V$ be a {  continuous} $\R^k$ invariant distribution. Then fix { an $\R^k$-periodic} point $p \in X$. {  Such points are dense by Lemma \ref{lem:dense-periodic}.}
Then choose 
 an Anosov element $a \in \R^k$ such that $\alpha(a) p = p$. {  Note that since the stabilizer of $p$ is a lattice in $\R^k$, such an Anosov element exists in every open cone}. 
 Then since $TX = T\mc O \oplus E^s_a \oplus E^u_a$, and $d\alpha(a)V(p) = V(p)$, $V(p)$ has a conmmon refinement with the stable and unstable splitting at $p$ since they are sums of the generalized eigenspaces for $d\alpha(a)$. So there exists corresponding subspaces of $V(p)$ such that $V(p) = V^0_a(p) \oplus V^s_a(p) \oplus V^u_a(p)$, and $V^*_a(p) = E^*_a(p) \cap V(p)$.

Now, since all distributions are continuously varying and this splitting holds at periodic orbits, since the periodic orbits are dense, $V$ splits everywhere as $V^0 \oplus V^s_a \oplus V^u_a$. This { procedure} can be repeated for another Anosov element $b$ to refine each new invariant distribution into the stable and unstable distributions for $b$. In particular, since each common stable manifold is either a coarse Lyapunov distribution or can be refined, the final refinement gives a subspace of each coarse Lyapunov distribution. In particular, since the dimension of the coarse Lyapunov distributions are assumed to be 1 for Cartan actions, either the distribution appears fully as part of the final splitting of $V$, or does not appear at all.
\end{proof}

\begin{definition}
\label{def:accessibility}
Let $\alpha : \R^k \curvearrowright X$ be a Cartan action. A {\it coarse Lyapunov path} based at $x \in X$ is a finite sequence $\rho = (x = x_0,x_1,\dots, x_n)$ such that $x_{i+1} \in W^{\beta_i}(x_i)$ for some coarse Lyapunov foliation $W^{\beta_i}$. $c(\rho) = n$ is called the {\it combinatorial length} of the path $\rho$, and $L(\rho) = \sum_{i=0}^{n-1} d_{W^{\beta_i}}(x_i,x_{i+1})$ is called the {\it geometric length} of $\rho$. $e(\rho) = x_n$ is called the endpoint of $\rho$. Let $\mc P_{c,L}^\alpha(x) = \set{ \rho \mbox{ based at } x : L(\rho) \le L, c(\rho) = c}$, and note the $\mc P_{c,L}^\alpha$ carries a canonical topology making it homeomorphic to the $L^1$-ball of radius $L$ in $\R^n$.

$\alpha$ is said to be {\it accessible} if $\bigcup_{c,L >0} \mc P_{c,L}^\alpha(x) = X$ for some (equivalently, every) $x\in X$.

$\alpha$ is said to be {\it continuously accessible} if for every $x \in X$, there is some $\ve,c,L > 0$ and a continuous map $\tau : B(x,\ve) \to \mc P_{c,L}^\alpha(x)$ such that $e(\tau(y)) = y$ for all $y \in B(x,\ve)$.
\end{definition}

In the case of flows ($\R$-actions), contact flows are the clearest examples of continuously accessible flows, since one may parameterize the flow direction by moving along $su$-quadrilaterals (via the {\it temporal distance function}, see, e.g., \cite[Appendices A \& B]{liverani04}). Then one can leverage the local product structure to build a continuous parameterization via paths. Baire arguments suggest that continuous accessibility is not too far from accessibility, see \cite[Proposition 7.2]{asv}. Following \cite[Section 3.4]{BPW} and \cite[Theorem 3.4]{ps97}, we will use the continuous accessibility property to get robustness of accessibility:

\begin{lemma}
If $\alpha_0 : \R^k \curvearrowright X$ is a continuously accessible Cartan action, then there exists a $C^1$-neighborhood $\mc U$ of $\alpha_0$ in the space of $C^1$ $\R^k$-actions such that every $\alpha \in \mc U$ is an accessible Anosov action.
\end{lemma}

\begin{proof}
Notice that if $\alpha$ is sufficiently close to $\alpha_0$, then the corresponding coarse Lyapunov foliations are also close. In particular, given a coarse Lyapunov path $\rho$ for $\alpha_0$, one may find a corresponding coarse Lyapunov path for $\alpha$ with the same combinatorial pattern, and same lengths of legs. That is, we associate a map $\pi : \mc P_{c,L}^{\alpha_0}(x) \to \mc P_{c,L}^\alpha(x)$. Therefore, we may consider the following map from $B(x,\ve)$ to $X$: $q : y \mapsto e(\pi(\tau(y)))$. Then $q$ is $C^0$-close to $\id$.

Let $S$ denote the sphere of radius $\ve/2$ in $B(x,\ve)$. Then by choosing the pertubation small enough, we may assume that $0 \not\in q(S)$, so $q : S \to B(x,\ve) \setminus \set{x}$ is a $C^0$-perturbation of the identity and hence has the same degree after projecting back to $S$ along rays in some fixed coordinate chart. It follows that $q(B(x,\ve/2))$ contains a neighborhood of $x$, and hence the accessibility class of $x$. Since this is true for every $x$, it follows that the action $\alpha$ is accessible. 
\end{proof}

\begin{corollary}
\label{cor:accessible}
If $\alpha_0 : \R^k \curvearrowright X$ is a Cartan action defined as a $k$-fold product of continuously accessible Anosov flows on 3-manifolds, there exists a $C^1$-neighborhood $\mc U$ of $\alpha_0$ in the space of $C^1$ $\R^k$-actions such that every $\alpha \in \mc U$ is an accessible Anosov action.
\end{corollary}


\section{Rank one factors and hyperplanes}
\label{sec:rankone-hyperplanes}

{In this section, we wish to establish a way of detecting factors by considering the derivative cocycles along coarse Lyapunov distributions. To that end, we first establish the way in which coarse distributions interact with factors.

\begin{lemma}
\label{lem:factors-anosov}
Let $\alpha : \R^k \curvearrowright X$ be a cone transitive, $C^1$ Cartan action, and $\psi_t : \R \curvearrowright Y$ be a rank one factor action. Then $\psi_t$ is an Anosov flow on a 3-manifold or a transitive circle flow. Furthermore, if $\psi_t$ is Anosov, there exists a unique pair of coarse Lyapunov distributions $E^{\chi_+}$ and $E^{\chi_-}$ such that $d\pi(E^{\chi_+}) = E^s$ and $d\pi(E^{\chi_-}) = E^u$.
\end{lemma}

To prove Lemma \ref{lem:factors-anosov}, we use the following criterion established by Mane. Recall that a flow $\psi_t : \R \curvearrowright Y$ is {\it quasi-Anosov} if for any vector $v \in TY$, $\set{\norm{d\psi_t(v)} : t \in \R}$ is unbounded. Note that this implies every periodic orbit is hyperbolic, and hence each periodic point has well-defined stable and unstable manifolds. 

\begin{theorem}[Corollary 1, \cite{maneQuasi}]
If $\psi_t$ is quasi-Anosov and for every pair of periodic points $p,q$ of $\psi_t$, $\dim(W^s(p)) = \dim(W^s(q))$, then $\psi_t$ is Anosov.
\end{theorem}

\begin{proof}[Proof of Lemma \ref{lem:factors-anosov}]
We first show that $\psi_t$ is quasi-Anosov. Note that the distribution $E(x) := \ker d\pi(x)$ is $\alpha$-invariant, so by Lemma \ref{lem:invariant-bundles}, $E$ is a sum of coarse Lyapunov subbundles and a subbundle of the $\alpha$-orbit distribution. By the intertwining property $\pi \of \alpha(a) = \psi_{\sigma(a)} \of \pi$, it follows that the subbundle of the $\alpha$-orbit distribution is exactly the $\ker \sigma$-orbit distribution.

 Indeed, fix some $y \in Y$, $v \in T_yY$, and pick some (aribtrary) $x \in \pi^{-1}(y)$. Then if $v$ is not tangent to the $\psi_t$-orbit of $y$, any lift of $v$ will not be tangent to the $\alpha$-orbit of $x$ (a lift must exist by the submersion property). In particular, there exists an Anosov element $a$ for which $d\alpha(ka)v$ grows exponentially in $k$, and $v$ has a nontrivial coarse Lyapunov distribution which is not in $\ker d\pi$ at any point. Therefore,  $d\psi_{k\sigma(a)}(v)$ grows exponentially in $k$. It follows that $\psi_t$ is quasi-Anosov.
 
 We now check that the stable and unstable manifolds at each periodic orbit are of the same dimension. Indeed, fix an Anosov element $a$ such that $\sigma(a) = 1$. Such a choice is possible by first choosing an Anosov element $a_0$ such that $\sigma(a_0) \not= 0$. Since the set of Anosov elements is open and $\ker \sigma$ is codimension 1, such a choice is possible. Then simply let $a = \dfrac{1}{\sigma(a_0)}a_0$. We claim that $\dim(E^s_{\psi}(p))$ is always $\#\set{\chi \in \Delta : E^\chi \subset W^s_{\alpha(a)} \mbox{ and } E^\chi \not\subset \ker d\pi}$. Indeed, fix $x \in X$ such that $\pi(x)$ is a $\psi_t$-periodic orbit. Each distribution $E^\chi$ is 1-dimensional by the Cartan condition, and $d\pi$ will push the sum of the distributions to an invariant contracting distribution. Since at a periodic orbit, the decomposition into a contracting distribution, orbit distribution and expanding distribution is unique, the dimension is as described. 
 
 Finally, we show that either 0 or exactly two coarse Lyapunov distributions descend. Since we have shown that the flow is Anosov, we know that each coarse distribution $E^\chi$ of $\alpha$ which is not in $\ker d\pi$ must descend to a subspace of either the stable or unstable distribution of a distinguished Anosov element $a$. By the intertwining property $\pi \of \alpha(a) = \psi_{\sigma(a)} \of \pi$, if $E^\chi$ is contracted under $a$, it is also uniformly contracted under $a + b$ for any $b \in \ker \sigma$. Hence the set of contracting elements for $E^\chi$ is exactly a half space determined by $\ker \sigma$. If $E^\chi$ and $E^{\chi'}$ are distinct coarse Lyapunov distributions, there must exist an Anosov element which expands one and contracts the other. Thus, since we have determined that the set of contracting elements must be exactly one of two half spaces, there can only be at most 2 coarse Lyapunov distributions which descend to the factor. It cannot be exactly one since every periodic orbit would be attracting, a contradiction on a compact space. The result follows.
\end{proof}
}

\begin{lemma}
\label{lem:collapse-coarses}
Let $\R^k \curvearrowright X$ be a {  cone} transitive, $C^1$ Cartan action and $\psi_t : \R \curvearrowright Y$ be a $C^1$ rank one factor, with corresponding homomorphism $\sigma : \R^k \to \R$ and projection map $\pi : X \to Y$. Assume that $E^\chi$ is a coarse Lyapunov distribution. Then if $E^\chi  \cap \ker d\pi = \set{0}$ at some $x \in X$, there exists a continuous metric on $E^\chi$ such that for every $a \in \ker \sigma$, $da|_{E^\chi}$ is an isometry.
\end{lemma}

\begin{proof}

By Lemma \ref{lem:factors-anosov}, $E^\chi$ must descend to either the stable or unstable distribution of $\psi_t$. 
Choose any metric $\norm{\cdot}_Y$ on $Y$ and if $v \in E^\chi$, let $\norm{v}_{E^\chi} := \norm{d\pi(v)}_Y$. 
By construction, $\norm{\cdot}_{E^\chi}$ is continuous, and if $a \in \ker \sigma$, the action of $a$ on $Y$ is trivial, so:

\[\norm{d\alpha(a)(v)}_{E^\chi} = \norm{d\pi da(v)}_Y = \norm{d\pi(v)}_Y = \norm{v}_E. \] 
\end{proof}

\begin{corollary}
\label{cor:no-rank1-criteria}
Let $k \ge 2$ and $\alpha : \R^k \curvearrowright X$ be a cone transitive, accessible, $C^r$ Cartan action with such that for every $\beta \in \Delta$ and $a \in \R^k \setminus \set{0}$, there exists $x \in X$ such that $\displaystyle\lim_{n \to \infty} \frac{1}{n}\log \norm{d(na)|_{E^\beta}(x)} \not= 0$. Then $\alpha$ has no nontrivial $C^r$ rank one factors.
\end{corollary}

\begin{proof}
Our assumption implies that there is no subgroup of $\R^k$ which acts isometrically on $E^\beta$ for any continuous metric. Thereofore, if $\pi : X \to Y$ determines a rank one factor, then $E^\beta \subset \ker d\pi$ for every coarse Lyapunov distrubiton by Lemma \ref{lem:collapse-coarses}. Therefore, $\pi^{-1}(x)$ contains all coarse Lyapunov foliations, and since the action is accessible, $\pi^{-1}(x) = X$. That is, $\pi$ is a projection onto a point and there are no nontrivial rank one factors.
\end{proof}

\begin{corollary}
\label{cor:non-homogeneous}
Under the same assumptions as Corollary \ref{cor:no-rank1-criteria}, the action $\alpha$ is not homogeneous.
\end{corollary}

\begin{proof}
This follows almost immediately. Notice that the derivative of a homogeneous action is always determined by the adjoint representation. Since the coarse Lyapunov distribution is 1-dimensional, it must be spanned by a joint eigenvector of the $\R^k$-action. Since every functional from $\R^k$ to $\R$ has a nontrivial kernel, there must exists some $a$ such that $\norm{da|_{E^\chi}} =1$ with respect to any right-invariant metric. This is incompatible with the assumptions.
\end{proof}

\section{Construction of the example}

Let $f_s : Y_1 \to Y_1$ and $g_t : Y_2 \to Y_2$ be continuously accessible Anosov flows { with $\dim(Y_i) = 3$, $i = 1,2$}. 
Let $M = Y_1 \times Y_2$ and consider the product action $\alpha_0 : \R^2 \curvearrowright M \times M$ defined by

\[\alpha_0(s,t)(x_1,x_2) = (f_s(x_1),g_t(x_2)). \]

Then $\alpha_0$ is (totally) Cartan with four coarse Lyapunov distributions, $W^{\pm \chi_1}$, $W^{\pm \chi_2}$ corresponding to the stable and unstable bundles in each factor of the action. That is, $E^{\chi_1} = E^s_f \times \set{0}$, $E^{-\chi_1} = E^u_f \times \set{0}$, $E^{\chi_2} = \set{0} \times E^s_g$ and $E^{-\chi_2} = \set{0} \times E^u_g$. Furthermore, the elements $(\pm 1,\pm 1) \in \R^2$ are Anosov elements of the action. Let $\ve_1$ be such that if $F : M \to M$ is such that $d_{C^1}(F,a) < \ve_1$ for $a = (\pm 1,\pm 1)$, then $F$ acts normally hyperbolically with respect to a nearby foliation, and nearby distributions (such a $\ve_1$ exists by Hirsch-Pugh-Shub normal hyperbolicity theory). Let $\ve_0$ be as in Lemma \ref{lem:build-time-change} and choose $\delta < \min\set{\ve_0/4,\ve_1/100}$ and points $p_1,p_2 \in Y_1$, $q_1,q_2 \in Y_2$ which lie on distinct periodic orbits of $f_t$ and $g_s$, respectively. We may assume that a continuous Riemannian metric on $Y_1$ and $Y_2$ has been chosen so that there exist coefficients $\lambda_{i,*}$ and $\mu_{i,*}$, $i = 1,2$ and $* = s$ or $u$ such that

\[ \norm{df_t|_{E^*_f}}(p_i) = e^{t\lambda_{i,*}}, \qquad \norm{dg_s|_{E^*_s}}(q_i) = e^{s\mu_{i,*}} \qquad \mbox{for }i=1,2, \; * = s \mbox{ or }u. \]

 Finally, pick functions $u_i : Y_i \to \R$, $i = 1,2$ such that

\begin{enumerate}
\item $u_i$ is $C^\infty$, $i = 1,2$
\item $u_1(f_t(p_1)) \equiv u_2(g_s(q_1)) \equiv \delta$ for all $s,t \in \R$
\item $u_1(f_t(p_2)) \equiv u_2(g_s(q_2)) \equiv -\delta$ for all $s,t \in \R$
\item $\abs{u_1},\abs{u_2} \le 2\delta$
\end{enumerate}

Such functions $u_i$ generate cocycles $\theta_i$ over the flows $f_t$ and $g_s$ via the formula $\theta_1(t,x) = \int_0^t u_1(f_\tau(x)) \, d\tau$ and $\theta_2(s,x) = \int_0^t u_2(g_\tau(x)) \, d\tau$. 

Then define a cocycle $\beta$ over $\alpha_0$ by:

\[ \beta(s,t;x) = (s - \theta_2(t,x_2),t - \theta_1(s,x_1)). \]

One easily verifies that $\beta$ satsfies property \eqref{eq:cocycle-eqn} for the action $\alpha_0$. Furthermore, by the smallness assumption on $u_i$, $i = 1,2$, the cocycle $\beta$ also satisfies the assumptions of Lemma \ref{lem:build-time-change}. Let $\alpha$ be the corresponding time change of $\alpha_0$. We may further assume that $\delta$ is chosen small enough so that $\alpha \in \mc U$, where $\mc U$ is the neigbhorhood in Corollary \ref{cor:accessible}.

\begin{theorem}
$\alpha$ is a $C^\infty$ Cartan action without rank one factors, and which is not homogeneous.
\end{theorem}

\begin{proof}
First, notice that $(\pm 1,\pm 1)$ are still Anosov elements since our cocycle $\beta$ was sufficiently close to $\id$. Therefore, we have the same indexing set for the coarse Lyapunov foliations $\set{\pm \chi_1,\pm \chi_2}$, even though their distributions and foliations may be perturbed.

By Corollary \ref{cor:accessible}, $\alpha$ us accessible. So by Corollaries \ref{cor:no-rank1-criteria} and \ref{cor:non-homogeneous}, it suffices to show that given any $\chi \in \Delta$, every $a \in \R^2 \setminus \set{0}$ has some point $x \in X$ such that $\displaystyle\lim_{n\to \infty} n^{-1}\log \norm{d\alpha(na)|_{E^{\chi}}(p)} \not= 1$. We work with $E^{\chi_1}$, since all other coarse Lyapunov distributions will have a symmetric argument. Consider the derivatives of $a$ at the points $x = (p_1,q_2) \in M \times M$ and $y = (p_2,q_1) \in M \times M$. By assumption, for fixed $(s,t)$ we may explicitly compute $\beta$ near near $x$, $\beta(s,t;x) = (s+\delta t,t-\delta s)$. Therefore, with $(s,t)$ fixed and $x'$ near $x$, 

\[ \varphi(s,t;x') = \frac{1}{1+\delta^2}(s -\delta t,t + \delta s). \] 

Fix $a = (s,t)$, so that the function $\varphi$ is constant in a neighborhood of the $p_1$-orbit. Therefore, since the time change is constant in a neigbhorhood, $E^{\chi_1}$ is exactly a coarse Lyapunov distribution for $\alpha$ along the orbit $x$, as it is an invariant distribution transverse to $\mc O$ at $x$. Denote $(s',t') = \varphi(s,t;x)$. Now, we get that if $v \in E^{\chi_1}(x)$,

\[ d\alpha(s,t)v = d\alpha_0(s',t')v = e^{\lambda_{1,u}(s-\delta t)/(1+\delta^2)}. \]

By a symmetric computation, for fixed $a = (s,t)$, $\varphi(s,t;y') = \frac{1}{1+\delta^2}(s+\delta t,t  - \delta s)$ with $y'$ near $y$. Therefore, if $v \in E^{\chi_1}(y)$

\[ d\alpha(a)v = e^{\lambda_{2,u}(s+\delta t)/(1+\delta^2)}. \]

It is not possible that $s - \delta t = s + \delta t = 0$ unless $s=t = 0$. Therefore, no non-identity element has zero exponents for $\chi_1$ at every $x \in X$. We may repeat this process for $-\chi_1$ and $\pm \chi_2$. Then by Corollary \ref{cor:no-rank1-criteria}, the action $\alpha$ has no rank one factors, and Corollary \ref{cor:non-homogeneous}, the action is not homogeneous.
\end{proof}

\section{Remarks on the example and conjecture}
\label{sec:updates}

We begin by briefly noting that this example was discovered in the context of several other unexpected examples, and is indirectly related to them. In \cite{spatzier-vinhage}, such examples are discussed at length. Another important example of a $\Z^2$ action with nontrival coexistence of rigidity and flexibility properties was recently constructed by Damjanovic, Wilkinson and Xu \cite{DWX}.

The Katok-Spatzier conjecture can be reformulated in a variety of settings. One way to adjust the conjecture is to strengthen the assumptions. In the formulation of Conjecture \ref{conj:kat-spa}, one assumes that the $\R^k$-action has no $C^\infty$ rank one factors. Notice every $C^\infty$ rank one factor is a continuous rank one factor, and in the measure-preserving setting, every continuous rank one factor is a measurable rank one factor. Therefore, one may consider asking the action to have no continuous, or no measurable rank one factors (with respect to an invariant volume) in order to guarantee rigidity.

One should expect that these examples remain counterexamples with to such revisions of Conjecture \ref{conj:kat-spa}. Indeed, one may see the destruction of a measurable rank one factor when one only destroys one smooth factor. When one uses the cocycle $\beta(s,t;x) = (s,t-\theta_1(s,x_1))$, we may explicitly compute the corresponding function $\varphi(s,t;x) = (s,t+\theta_1(s,x_1))$. Then the time change $\alpha$ induced by $\varphi$ contains a skew product action: the horizontal direction $(s,0)$ is exactly a skew product. Skew products determined by cocycles not cohomologous to a constant are ergodic. Combined with the following, this shows that the projection onto the second factor of $M = Y_1 \times Y_2$ is no longer a rank one factor.

\begin{lemma}
\label{lem:lines-ergodic}
A measure-preserving action $\alpha : \R^2 \curvearrowright (X,\mu)$ has a nontrivial measurable rank one factor if and only if there exists a line $L \subset \R^2$ such that the restriction of the action to $L$ is not ergodic.
\end{lemma}

\begin{proof}
First, assume that there exists a rank one factor $\psi_t : (Y,\nu) \to (Y,\nu)$ determined by a measurable map $\pi : X \to Y$ such that $\pi_*\mu = \nu$ and homomorphism $\sigma : \R^2 \to \R$. Then if $L = \ker \sigma$, $L$ acts trivially on $Y$. Since $Y$ is nontrivial, any function on $X$ defined by $\psi \of \pi$, for some measurable function $\psi : Y \to \R$ is invariant under $L$. Since $Y$ is not trivial, there exist nontrivial $L$-invariant functions, and the $L$-action is not ergodic.

Now, assume that the restriction of the action to $L$ is  not ergodic. By the ergodic decomposition theorem, there exists a $\mu$-almost-everywhere defined map to the space of $L$-invariant measures $\mc M(L)$, $\phi : (X,\mu) \to (\mc M(L),\nu)$ such that for any $f \in L^1(X,\mu)$,

\[ \lim_{T\to \infty} \dfrac{1}{T} \int_0^T f(\psi_{t\ell}(x)) \, dt = \int_X f \, d\phi(x). \]

By construction, the action of $\R^k$ descends to $(\mc M(L),\nu)$ via $\alpha(a)m = a_*m$, and the $L$-action is trivial. That is, $\phi$ determines a measurable rank one factor of the $\R^k$ action.
\end{proof}

Investigation of these examples as measure preserving transformations, in particular their ergodic and statistical properties, is ongoing.

Another way to account for these new examples would be to ask for no rank one factors, even after the modifications used to produce these new examples.

\begin{question}
If no $C^\infty$ time change of a cone transitive, $C^\infty$ Anosov action $\alpha : \R^k \curvearrowright X$ has a $C^\infty$ rank one factor, is $\alpha$ $C^\infty$ conjugate to an algebraic system?
\end{question}

Another interpretation would be to ignore the parameterization of orbits induced by the action altogether, and consider only the orbit foliations.

\begin{question}
Assume that $\alpha : \R^k \curvearrowright X$ is a cone transitive, $C^\infty$, Anosov action, and that there does not exist a nontrivial $C^\infty$ flow $\psi_t : Y \to Y$ and submersion $\pi : X \to Y$ such that $\pi(\alpha(\R^k)x) = \set{\psi_t(\pi(x)) : t \in \R}$ for all $x \in X$. Is $\alpha$ $C^\infty$ conjugate to an algebraic system?
\end{question}

In view of Remark \ref{rmk:params-matter}, allowing for time changes is a more accurate reflection of the spirit of the rigidity program. A time change is determined by a cocycle over the new action. If an action $\alpha$ is cocycle-rigid, as are the homogeneous actions without rank one factors, any $C^r$ time change should be smoothly conjugate to a linear time change of $\alpha$. Therefore, even if one allows for a time change, if rigidity holds, that time change would be trivial. Thus, one still obtains a smooth conjugacy with the original action.

Finally, when the action is cone transitive and {\it totally} Cartan (see Definitions \ref{def:anosov} and \ref{def:cartan}), the main theorem of \cite{spatzier-vinhage} implies that if the action has no rank one factor, it is $C^\infty$ conjugate to an algebraic system. This immediately implies that the examples here are {\it not} totally Cartan (which is also observable directly from computations), and motivates the following

\begin{question}
Let $\R^k \curvearrowright X$ be a cone transitive, Cartan action. Is there a $C^\infty$ time change of the action which is totally Cartan?
\end{question}

\bibliographystyle{plain}
\bibliography{rigidity}

\end{document}